\documentclass[12pt, 14paper,reqno]{amsart}
\usepackage{amsmath,amsfonts,amssymb}

\theoremstyle{plain}
\newtheorem{theorem}{Theorem}
\newtheorem{lemma}{Lemma}

\theoremstyle{proof}
\theoremstyle{definition}

\theoremstyle{remark}

\theoremstyle{lamma}

\numberwithin{equation}{section}
\numberwithin{lemma}{section}
\numberwithin{theorem}{section}

\usepackage{amsmath}
\usepackage{amsfonts}   
\usepackage{amssymb}
\usepackage{amssymb, amsmath, amsthm}
\usepackage[breaklinks]{hyperref}

\theoremstyle{thmrm}

\begin{document}
\title[Class numbers of quadratic fields]{Divisibility of class numbers of certain families of quadratic fields}
\author{Azizul Hoque and Kalyan Chakraborty}
\address{Azizul Hoque @Azizul Hoque Harish-Chandra Research Institute,
Chhatnag Road, Jhunsi,  Allahabad 211 019, India.}
\email{ azizulhoque@hri.res.in}

\address{Kalyan Chakraborty @Kalyan Chakraborty, Harish-Chandra Research Institute,
Chhatnag Road, Jhunsi, Allahabad 211 019, India.}
\email{kalyan@hri.res.in}

\keywords{Quadratic fields, Discriminant, Class number, Hilbert class field}
\subjclass[2010] {Primary: 11R29, Secondary: 11R11}
\maketitle
\begin{abstract}
We construct some families of quadratic fields whose class numbers are divisible by $3.$ The main tools used are a trinomial introduced by Kishi and a parametrization of Kishi and Miyake of a family of quadratic fields whose class numbers are divisible by $3.$ At the end we compute class number of these fields for 
some small values and verify our results. 
\end{abstract}

\section{Introduction}
The ideal class group or more precisely the class number of number fields is one of the most fundamental and mysterious objects associated with the field extensions.
Starting from Gauss, this area has attracted the attention of many researchers. It is well-known that there exist infinitely many quadratic fields
each with class number divisible by a given integer $g\geq 2.$ In particular, Nagell \cite{NA22} proved that there are infinitely many imaginary quadratic fields with class 
number divisible by a given integer $g\geq 2.$ On the other hand, for real quadratic field case, Honda \cite{HO68}, Yamamoto \cite{YA70}, Weinberger \cite{WE73} and Ichimura \cite{IC03} 
independently proved that there are infinitely many real quadratic fields each with class number divisible by a given integer $g\geq 2.$ In  recent years, the study is concentrating on  
characterising  such fields, i.e. each with class number divisible by a given integer $g\geq 2$. In this direction, Kishi and Miyake \cite{KM00} gave a parametrization of 
quadratic fields with class number divisible by $3$ and this enabled enlargement of  the list of families of quadratic fields with the divisibility properties. Chakraborty and
Murty \cite{CM03}, Kishi \cite{KI10}, Hoque and Saikia \cite{HS162} contributed some members to this list. 

In this paper we provide some infinite, simply parametrized new families of 
quadratic fields  with class number divisible by $3$. More precisely, we show that under certain conditions on the integers $a, b, m, n, p$ and $r$,
the class numbers of the fields $ \mathbb{Q}(\sqrt{3(4m^{3n}-k^2)})$, $\mathbb{Q}(\sqrt{-(m^2n^2\pm 4n)/3})$, $ \mathbb{Q}(\sqrt{-(3^mp^{2n}+r)})$, \hspace*{2mm} 
$\mathbb{Q}(\sqrt{3(a^{3n}-b^{2n})})$,\hspace*{2mm} $\mathbb{Q}(\sqrt{3(4a^{3n}-b^{2n})})$, \\ $\mathbb{Q}(\sqrt{-3(4m^3+1)})$, $\mathbb{Q}(\sqrt{3(2m^{3n}-1)})$
and $\mathbb{Q}(\sqrt{1-2m^3})$ are divisible by $3.$ We begin by fixing some notations.

{\bf{Notations}}: For a number field $K,$ $\Delta_K$ and $\mathcal{O}_K$ denote the discriminant and the ring of integers of $K,$ respectively. We denote by $N_{K/\mathbb{Q}}$ and $T_{K/\mathbb{Q}}$ the norm and trace map of a number field $K,$ respectively. For a non-square integer $d$, $h(d)$ denotes the class number of $\mathbb{Q}(\sqrt{d})$. For a prime number $p$ and an integer $n,$ $\mathit{v}_p (n)$ denotes the greatest exponent $\mu$ of $p$ such that $p^\mu \mid n$. For a polynomial $f,$ $S_{\mathbb{Q}}(f)$ denotes the splitting field of $f$ over $\mathbb{Q}.$

\section{Kishi's trinomial and class numbers of quadratic fields}
We begin with some lemmas and then recall a criterion for an extension to be unramified. Finally, we construct some 
quadratic fields each with class number divisible by $3$ using Kishi's trinomial.

Let $\alpha \in \mathcal{O}_K$ with $N_{K/\mathbb{Q}}(\alpha)\in \mathbb{Z}^3$ and
\begin{equation}\label{eq1}
 f_\alpha(X):=X^3-3[N_{K/\mathbb{Q}} (\alpha)]^{1/3}X-T_{K/\mathbb{Q}}(\alpha)
\end{equation}
The trinomial $f_\alpha(X)$ was introduced by Kishi in \cite{KI00}. 
We recall the following result of Kishi \cite{KI98}.
\begin{lemma}\label{lma2.1}
 Let $K=\mathbb{Q}(\sqrt{d}).$ Suppose $\alpha=\frac{a+b\sqrt{d}}{2} \in {\mathcal{O}}_K$ 
 with
 $N_{K/\mathbb{Q}}(\alpha)\in \mathbb{Z}^3.$ Then $f_{\alpha} (X)$  
 is 
 reducible over $\mathbb{Q}$ if and only if $\alpha$ is a cube in $K.$
\end{lemma}

Let $d(\ne 1, -3)$ be a square-free integer and
$$
\begin{displaystyle}
D = 
\begin{cases}
-d/3 \hspace*{2mm}\text{if } d \text{ is a multiple of  } 3, \\
  -3d \hspace*{4mm}\text{otherwise}. 
  \end{cases}
  \end{displaystyle}
$$
Let 
$K=\mathbb{Q}(\sqrt{d})$ and $L=\mathbb{Q}(\sqrt{D})$. Also,
$$
R_d:=\{\alpha \in\mathcal{O}_K: \alpha \text{ is not  a  cube in } K \text{ and} \ N_{K/\mathbb{Q}}(\alpha)\text{ is a  cube in } \mathbb{Z}\}
$$
and
$$
R_D:=\{\alpha \in\mathcal{O}_L: \alpha \text{ is not a cube in } L \text{ and } N_{L/\mathbb{Q}}(\alpha)\text{ is a cube in } \mathbb{Z}\}.
$$
It is clear that the subset $R_d$ (respectively $R_D$) contains all those units in $K$ which are not cubes in $K$ (respectively in $L$). Further let,
$$
R^*_d:=\{\alpha\in R_d: \gcd(N_{K/\mathbb{Q}}(\alpha), T_{K/\mathbb{Q}}(\alpha))=1\}
$$
and 
$$
R^*_D:=\{\alpha\in R_D: \gcd(N_{L/\mathbb{Q}}(\alpha), T_{L/\mathbb{Q}}(\alpha))=1\}.
$$
We can now recall a result of Kishi (\cite{KI00}, Proposition 6.5) which is one of the main ingredient for deriving our results.
\begin{lemma}\label{lma2.3}
 Let $\alpha \in R^*_D$ (resp. $\alpha \in R^*_d$). Then $S_{\mathbb{Q}}(f_\alpha)$ is an $S_3$-field containing $K=\mathbb{Q}(\sqrt{d})$ (resp. $L=\mathbb{Q}(\sqrt{D})$) which is a cyclic cubic extension of $K$ (resp. $L$) unramified outside $3$ and contains a cubic subfield $K'$ with $\mathit{v}_3(\Delta_{K'})\ne 5.$ Conversely, every $S_3$-field containing $K$ (resp. $L$) which is unramified outside $3$ over $K$ (resp. $L$) and contains a cubic subfield $K'$ satisfying $\mathit{v}_3(\Delta_{K'})\ne 5$ is given by $S_{\mathbb{Q}}(f_\alpha)$ with $\alpha \in R^*_D$ (resp. $\alpha \in R^*_d$). 
 \end{lemma}
The following result of Llorente and Nart (\cite{LN83}, Theorem 1) talks about ramification at the prime $p=3.$
 \begin{lemma}\label{lma2.5}
  Suppose that
  $$
  g(X):= X^3-aX-b\in \mathbb{Z}[X]
  $$
 is irreducible over $\mathbb{Q}$ and that either $\mathit{v}_3(a)<2$ or $\mathit{v}_3(b)<3$ holds. Let $\theta$ be a root of $g(X).$ 
 Then $3$ is totaly ramified in $\mathbb{Q}(\theta)/\mathbb{Q}$ if and only if one of the following conditions holds:
 \begin{enumerate}
  \item[(LN-1)] $1\leq \mathit{v}_3(b)\leq \mathit{v}_3(a),$
  \item[(LN-2)] $3\mid a, \ a\not\equiv3 \pmod 9, \ 3\nmid b \ and \ b^2\not\equiv a+1 \pmod 9,$
  \item[(LN-3)] $ a \equiv3 \pmod 9, \ 3\nmid b \ and \ b^2\not\equiv a+1 \pmod {27}.$
 \end{enumerate}
 \end{lemma}
Now we can proceed to our first result.
\begin{theorem}\label{thm2.1}
Let $m\equiv 0 \pmod 3$ be odd and $n$ be any positive integers. If $d_1$ is the square-free part of $3(4m^{3n}-k^2)$ with $k\equiv \pm 1\pmod{18}$ and $\gcd(m, k)=1,$ then $3| h(d_1).$
\end{theorem}
\begin{proof}
Let $D_1=-d_1/3, \hspace*{1mm} K_1=\mathbb{Q}(\sqrt{d_1})$ and $L_1 =\mathbb{Q}(\sqrt{D_1})$. Let $\alpha_1 \in \mathcal{O}_{L_1}$ so that  
 $$
 \alpha_1 = \frac{k+\sqrt{k^2-4m^{3n}}}{2}.
 $$
Then $T_{L_1/\mathbb{Q}}(\alpha_1)=k$ and $N_{L_1/\mathbb{Q}}(\alpha_1)=m^{3n}.$ Since $\gcd(m, k)=1,$ so that $$\gcd (T_{L_1/\mathbb{Q}}(\alpha_1), N_{L_1/\mathbb{Q}}(\alpha_1))=1.$$
 Now with respect to $\alpha_1 ,$
\begin{align*}
 f_{\alpha_1}(X): & =X^3-3[N_{L_1/\mathbb{Q}}(\alpha_1)]^{1/3}X-T_{L_1/\mathbb{Q}}(\alpha_1)\\
  & = X^3-3m^nX-k\\
  & \equiv X^3 + X+1 \pmod 2.
\end{align*}
Thus the polynomial $f_{\alpha_1} (X)$ is irreducible over $\mathbb{Z}_2$ and therefore it is irreducible over $\mathbb{Q}$. Therefore
by Lemma \ref{lma2.1}, $\alpha_1$ is not a cube in $L_1$ and thus $\alpha_1 \in R_{D_1}.$ Since $\gcd (T_{L_1/\mathbb{Q}}(\alpha_1), N_{L_1/\mathbb{Q}}(\alpha_1))=1,$ so that $\alpha_1\in R^*_{D_1}$ too. Therefore by Lemma \ref{lma2.3}, $S_{\mathbb{Q}}(f_{\alpha_1})$ is a cyclic cubic extension of $K_1$ which is unramified outside $3$.

Now we are left to show that $S_{\mathbb{Q}}(f_{\alpha_1})$ is unramified over $K_1$ at $3$ too.
The polynomial $f_{\alpha_1} (X)$ does not satisfy the condition (LN-1) as $\mathit{v}_3(1)=0$. Also $f_{\alpha_1}(X)$ does not 
satisfy the conditions (LN-2) and (LN-3) since $m\equiv 0 \pmod 3.$ Therefore by Lemma \ref{lma2.5}, $S_{\mathbb{Q}}(f_{\alpha_1})$ is unramified over $K_1$ at $3.$ Thus by Hilbert class field theory the class number of $K_1$ is divisible by $3.$
\end{proof}

\begin{theorem}\label{thm2.2}
 Let $m\equiv 0 \pmod 3$ and $n$ be any odd integers such that $\mathit{v}_3(n)=1.$ Then the class number of $\mathbb{Q}(\sqrt{-(m^2n^2\pm 4n)/3})$ 
 is divisible by $3.$
\end{theorem}

\begin{proof} 
Let $d_2$ be the square-free part of $-(m^2n^2\pm 4n)/3$ and $K_2= \mathbb{Q}(\sqrt{d_2}).$ Suppose that $D_2=-3d_2$ and $L_2=\mathbb{Q}(\sqrt{D_2})$. 
Let $\alpha_2 \in \mathcal{O}_{L_2}$ be so that
 $$
 \alpha_2 = \dfrac{m^2n\pm 2+m\sqrt{m^2n^2\pm 4n}}{2}.
 $$
 Then $T_{L_2/\mathbb{Q}}(\alpha)=m^2n\pm 2$ and $N_{L_2/\mathbb{Q}}(\alpha_2)=1$. Thus  again their gcd is $1$.
 
We can now have the cubic polynomial corresponding to such an $\alpha_2$:
\begin{align*}
 f_{\alpha_2}(X): & =X^3-3[N_{L_2/\mathbb{Q}}(\alpha_2)]^{1/3}X-T_{L_2/\mathbb{Q}}(\alpha_2)\\
  & = X^3-3X-(m^2n\pm 2)\\
  & \equiv X^3-X- 1 \pmod 2.
\end{align*}
Thus the polynomial $f_{\alpha_2} (X)$ is irreducible over $\mathbb{Z}_2$ and therefore it is irreducible over $\mathbb{Q}$. Therefore
by lemma \ref{lma2.1}, $\alpha_2$ is not a cube in $L_2$ 
and thus $\alpha_2\in R^*_{D_2}.$ Therefore by Lemma \ref{lma2.3}, $S_{\mathbb{Q}}(f_{\alpha_2})$ is a cyclic cubic extension of $K_2$ unramified outside $3.$

Now it remains to show that $S_{\mathbb{Q}}(f_{\alpha_2})$ is unramified over $K_2$ at $3.$
Since $m\equiv 0 \pmod 3$ and $\mathit{v}_3(n)=1,$ the polynomial $f_{\alpha_2} (X)$ does not satisfy the conditions (LN-1), (LN-2) and (LN-3). 
Therefore by Lemma \ref{lma2.5}, 
$S_{\mathbb{Q}}(f_{\alpha_2})$ is unramified over $K_2$ at $3$ too. Thus by Hilbert class field theory the class number of $K_2$ is divisible by $3.$
\end{proof}

We now give an extension of a result proved by Hoque and Saikia [\cite{HS151}, Theorem 3.1].

\begin{theorem}\label{thm2.3}
Let $m>1$ and $p$ be odd integers, and $n$ be any positive integer. Let $d_3$ be the square-free part of $-(3^m p^{2n}+r)$ with $r \in \{-2, 4\}.$ Then
$3| h(d_3).$
\end{theorem}

\begin{proof} Let $r=4$ and then $d_3\equiv 1\pmod 4$.
As before we set $D_3=-3d_3$ and $L_3=\mathbb{Q}(\sqrt{D_3})$. Then $L_3= \mathbb{Q}(\sqrt{3^{m+1}p^{2n}+12})$.
Choose $\alpha_3 \in\mathcal{O}_{L_3}$ by 
 $$
 \alpha_3:= \dfrac{3^m p^{2n}+2+3^{(m-1)/2} p^n\sqrt{3^{m+1} p^{2n}+12}}{2}.
 $$
 Then $T_{L_3/\mathbb{Q}}(\alpha_3)=3^mp^{2n}+2,\ N_{L_3/\mathbb{Q}}(\alpha_3)=1$ and thus $$\gcd(T_{L_3/\mathbb{Q}}(\alpha_3), N_{L_3/\mathbb{Q}}(\alpha_3))=1.$$
The cubic polynomial corresponding to $\alpha_3$ is:
\begin{align*}
 f_{\alpha_3}(X): & =X^3-3[N_{L_3/\mathbb{Q}}(\alpha_3)]^{1/3}X-T_{L_3/\mathbb{Q}}(\alpha_3)\\
  & = X^3-3X-3^mp^{2n}-2\\
  & \equiv X^3-X-1 \pmod 2.
\end{align*}
Thus the polynomial $f_{\alpha_3} (X)$ is irreducible over $\mathbb{Z}_2$ and therefore it is irreducible over $\mathbb{Q}$. Therefore
by Lemma \ref{lma2.1}, $\alpha_3$ is not a cube in $L_3$ 
and hence $\alpha_3\in R^*_{D_3}.$ Thus by Lemma \ref{lma2.3}, $S_{\mathbb{Q}}(f_{\alpha_3})$ is a cyclic cubic extension of $K_3$ unramified outside $3.$

We now claim that $S_{\mathbb{Q}}(f_{\alpha_3})$ is unramified over $K_3$ at $3$ too. The polynomial $f_{\alpha_3} (X)$ does not satisfy the conditions (LN-1), (LN-2) and (LN-3) as $\mathit{v}_3(3^mp^{2n}+2) = 0$ and $m>1$. 
Therefore by Lemma \ref{lma2.5}, we proof the claim. Thus by Hilbert class field theory the class number of $K_3$ is divisible by $3.$

Now let $r=-2.$ Then $3\nmid d_3$ and $d_3\equiv 3 \pmod 4$. 
Let us set $D_3' =-3d_3$ and $L'_3 =\mathbb{Q}(\sqrt{D_3'})$. Then $L'_3= \mathbb{Q}(\sqrt{3^{m+1}p^{2n}-6})$ and choose an 
element $\alpha'_3 \in \mathcal{O}_{L'_3}$ by 
 $$
 \alpha'_3:=3^m p^{2n}-1+ 3^{(m-1)/2} p^n\sqrt{3^{m+1} p^{2n}-6}.
$$
One can now complete the proof by a similar argument as in the previous case.
\end{proof}

\begin{theorem}\label{thm2.4}
Let $n>1$ be an odd integer and $a, b$ two more integers such that:
\begin{enumerate}
 \item[(C3.1)] $a \equiv 19 \pmod {30}$,
 \item[(C3.2)] $b\equiv 6 \pmod {15},$ and is coprime to $a$,
\end{enumerate} 
then $3$ divides the class number of $\mathbb{Q}(\sqrt{3(a^{3n}-b^{2n})}).$
\end{theorem}

\begin{proof}
Let $d_4$ be the square-free part of $3(a^{3n}-b^{2n})$ and $\ K_4 = \mathbb{Q}(\sqrt{d_4}).$ Suppose that $\ D_4 = -d_4/3$ and $L_4 = \mathbb{Q}(\sqrt{D_4}).$ 

Let $\alpha_4 \in \mathcal{O}_{L_4}$ so that  
 $$
 \alpha_4 = b^n+\sqrt{b^{2n}-a^{3n}}.
 $$
 Then $T_{L_4/\mathbb{Q}}(\alpha_4)=2b^n$ and $N_{L_4/\mathbb{Q}}(\alpha_4)=a^{3n}.$ 
Then 
\begin{align*}
 f_{\alpha_4}(X): & =X^3-3\big(N_{L_4/\mathbb{Q}}(\alpha_4)\big)^{1/3}X-T_{L_4/\mathbb{Q}}(\alpha_4)\\
  & = X^3-3a^nX-2b^n.
\end{align*}
By the conditions (C3.1) and (C3.2), we have 
$$f_{\alpha_4}(X)\equiv X^3+3X-2\pmod 5.$$
Clearly the polynomial $f_{\alpha_4}(X)$ is irreducible over $\mathbb{Z}_5$ and hence it is irreducible over $\mathbb{Q}$ too.
Thus  $\alpha_4$ is not a cube in $L_4$ by Lemma \ref{lma2.1} and hence $\alpha_4\in R_{D_4}.$
The conditions (C3.1) and (C3.2) entail $\gcd(N_{L_4/\mathbb{Q}} (\alpha), T_{L_4/\mathbb{Q}}(\alpha_4))=1$ and therefore by 
Lemma \ref{lma2.3}, $S_{\mathbb{Q}}(f_{\alpha_4})$ is a cyclic cubic extension of $K_4$ unramified outside $3.$

Now we are left to show that $S_{\mathbb{Q}}(f_{\alpha_4})$ is unramified over $K_4$ at $3$ also. The polynomial $f_{\alpha_4} (X)$ does not satisfy the condition (LN-1) since $n>1, \ a \equiv 1 \pmod 3$ and $b\equiv 0 \pmod 3.$ 
Again  $3\mid 2b^n$ by the condition (C3.2), and thus $f_{\alpha_4}
(X)$ satisfies none of the conditions (LN-2) and (LN-3). 
Therefore by Lemma \ref{lma2.5}, $S_{\mathbb{Q}}(f_{\alpha_4})$ is unramified over $K_4$ at $3.$ 
Thus by Hilbert class field theory the class number of $K_4$ is divisible by $3.$
\end{proof}

\begin{theorem}\label{thm2.5}
Let $n>1$ be an integer and $a, b$ two more integers satisfying:
\begin{enumerate}
 \item[(C3.3)] $\gcd(a, b)=1$;
 \item[(C3.4)] $a\equiv 1 \pmod  3$ and is odd;
 \item[(C3.5)] $b\equiv 0 \pmod 3$ and is odd;
\end{enumerate} 
Then $3$ divides the class number of $\mathbb{Q}(\sqrt{3(4a^{3n}-b^{2n})})$.
\end{theorem}
\begin{proof}
Let $d_5$ be the sqaure-free part of $3(4a^{3n}-b^{2n})$ and $K_5=\mathbb{Q}(\sqrt{d_5}).$ Suppose that $D_5=-d_5 /3$ and $L_5=\mathbb{Q}(\sqrt{D_5})$. 
 Let $\alpha _5 \in \mathcal{O}_{L_5}$ be of the form 
 $$\alpha _5 = \frac{b^n+\sqrt{b^{2n}-4a^{3n}}}{2}.$$
 Then $T_{L_5 /\mathbb{Q}}(\alpha _5)=b^n$ and $N_{L_5/\mathbb{Q}}(\alpha _5)=a^{3n}$ and the 
corresponding polynomial 
\begin{align*}
 f_{\alpha _5}(X): & =X^3-3[N_{L_5/\mathbb{Q}}(\alpha _5)]^{1/3}X-T_{L_5/\mathbb{Q}}(\alpha _5)\\
  & = X^3-3a^nX-b^n.
\end{align*}
The conditions (C3.4) and (C3.5) would imply $f_{\alpha _5} (X)\equiv X^3-X-1\pmod 2.$ Therefore $f_{\alpha _5} (X)$ is irreducible over $\mathbb{Z}_2$ and thus it is 
irreducible over $\mathbb{Q}$.
Thus by Lemma \ref{lma2.1}, $\alpha _5$ is not a cube in $L_5$ and hence $\alpha_5\in R_{D_5}$. By (C3.3), $\gcd(N_{L_5/\mathbb{Q}} (\alpha _5), T_{L_5/\mathbb{Q}}(\alpha _5))=1$ and therefore by 
Lemma \ref{lma2.3}, $S_{\mathbb{Q}}(f_{\alpha_5})$ is a cyclic cubic extension of $K_5$ unramified outside $3.$

It remains to prove that $S_{\mathbb{Q}}(f_{\alpha_5})$ is unramified over $K_5$ at $3$ too.
Clearly  $\mathit{v}_3(b^n)>\mathit{v}_3(3a^n)=1$ owing to (C3.4) and (C3.5) and thus $f_{\alpha _5} (X)$ does not satisfy the condition (LN-1). 
Also $3\mid b^n$ due to (C3.5) and thus $f_{\alpha _5}
(X)$ does not satisfy the conditions (LN-2) and (LN-3). Therefore by Lemma \ref{lma2.5},
$S_{\mathbb{Q}}(f_{\alpha_5})$ is unramified over $K_5$ at $3.$ 
Thus by Hilbert class field theory the class number of $K_5$ is divisible by $3.$
\end{proof}

\section{Some more families of quadratic fields}
In this section, we shall use a result of Kishi and Miyake \cite{KM00} for proving the first of the two theorems.  Let us first recall Kishi-Miyake parametrization.
\begin{lemma}\label{lma3.1}
Let $u$ and $v$ be two integers and
 \begin{equation}\label{eq6}
f_{u, v}(x)=x^3-uvx-u^2.
\end{equation}
If 
 \begin{enumerate}\label{ctn3.1}
  \item[(KM-1)] $u$ and $v$ are relatively prime;
  \item[(KM-2)] $f_{u, v}(x)$ is irreducible over $\mathbb{Q}$;
  \item[(KM-3)] discriminant $D_{f_{u, v}}$ of $f_{u, v}(x)$ is not a perfect square in $\mathbb{Z}$;
  \item[(KM-4)] one of the following conditions hold:
  \begin{enumerate}
   \item[(a)] $3\nmid v,$
   \item[(b)] $3\mid v,\ uv\not\equiv 3\pmod 9$ and $u\equiv v\pm 1 \pmod 9 ,$
   \item[(c)] $3\mid v, \ uv\equiv 3 \pmod 9 $ and $u\equiv v\pm 1 \pmod { 27} ,$
  \end{enumerate}
   \end{enumerate}
then $3$ divides the class number of $\mathbb{Q}(\sqrt{D_{f_{u, v}}})$. 
Conversely, every quadratic number field $\mathbb{Q}(\sqrt{D_{f_{u, v}}})$ with class number divisible by $3$ arises in the above way from 
a suitable choices of integers $u$ and $v$.
\end{lemma}
We use this to prove:
\begin{theorem}\label{thm3.1}
Let $m$ be an odd positive integer. 
\begin{enumerate}
 \item[(I)] If $m\equiv 0 \pmod 3,$ then $3$ divides the class number of the field $\mathbb{Q}(\sqrt{-3(4m^3+1)})$.
 \item[(II)] If $m\equiv 4 \pmod {15},$ then $3$ divides the class number of the field $\mathbb{Q}(\sqrt{3(2m^{3n}-1})$ for any odd integer $n\geq 3$. 
 \end{enumerate}
\end{theorem}
\begin{proof} We prove (I) and outline the proof of (II) as in most aspects these are very similar to the proof of (I).

Let us put $u=-1$ and $v=3m$ in (\ref{eq6}). Then 
\begin{equation*}
f_{-1, 3m}(X)=X^3+3mX-1
\end{equation*} 
and $D_{f_{-1, 3m}} = 9d$ with $d =-3(4m^3+1)$.

Clearly $u$ and $v$ are relatively prime, and $D_{f_{-1, 3m}}\neq\square$ in $\mathbb{Z}$. 
Now 
$$
f_{-1, 3m}(X)\equiv X^3-X-1 \pmod 2
$$ 
as $m$ is odd. Thus $f_{-1, 3m}(X)$ is irreducible over $\mathbb{Z}_2$ and therefore it is irreducible over $\mathbb{Q}$. 
Again $uv=-3m\equiv 0 \pmod 9$ (as $m\equiv 0 \pmod 3$). Furthermore $v-1=3m-1\equiv-1 \pmod 9 \equiv u \pmod 9$. Therefore $3$ divides class number of 
$\mathbb{Q}(\sqrt{-3(4m^3+1)})$ by invoking Lemma \ref{lma3.1}.

Finally to prove (II), we put $u=2$ and $v=3m^n$ in (\ref{eq6}). The irreduciblity of $f_{2, 3m^n}(X)$ 
follows from the 
irreducibility of $f_{2, 3m^n}(X)$ over
$\mathbb{Z}_5$. The condition (KM-3) holds since the discriminant of $f_{2,3m^n}(X)$ is $144d$ with $d=3(2m^{3n}-1)\equiv 3\pmod 4.$ Moreover we can show that the condition $(b)$ of (KM-4) holds.
\end{proof}

We conclude this section by providing another family of quadratic fields whose class number is divisible by $3$ by actually producing an element of order $3$.
\begin{theorem}\label{thm3.2}
 The class number of $\mathbb{Q}(\sqrt{1-2m^3})$ is divisible by $3$ for any odd integer $m>1.$
\end{theorem}
\begin{proof}
Let $d =1-2m^3$ and $K = \mathbb{Q}(\sqrt{d})$. Thus $d\equiv 3 \pmod 4$. 
Let $\alpha\in \mathcal{O}_K$ be of the form $\alpha=1+\sqrt{1-2m^3}$. Then $N_{K/\mathbb{Q}}(\alpha)=2m^3$. 

Suppose that $p_j$ is a prime factor of $m$. Then the Kronecker symbol $\left(\frac{d}{p_j}\right)=1$. Thus we can write
$$
(p_j)=\mathcal{P}_j\mathcal{P}'_j
$$ 
with distinct prime ideals $\mathcal{P}_j$ and $\mathcal{P}'_j$ in $\mathcal{O}_K$ which are conjugate to each other. Furthermore 
$$
(2)=\mathcal{P}^2  \hspace*{2mm}\mbox{with} \hspace*{2mm}\mathcal{P}=(2, 1+\sqrt{d}).
$$
We may express the prime ideal decomposition of $(\alpha)$ as
$$(\alpha) = \mathcal{P}\prod_j\mathcal{P}^{t_j},$$
because $\alpha$ is not divisible by any rational integers except $\pm 1$. 
Then $N_{K/\mathbb{Q}}((\alpha))=2\prod p_j^{t_j}$ with $p_j=N_{K/\mathbb{Q}}(P_j)$ and thus $3\mid t_j.$ Therefore 
$$\big(\mathcal{P}\prod_j\mathcal{P}_j^{t_j/3}\big)^3=(2)\mathcal{P}\prod_j \mathcal{P}_j^{t_j}=(2)(\alpha),$$
which is principal in $\mathcal{O}_K$. If $\langle\mathcal{I}\rangle$ is an ideal class containing $\mathcal{P}\prod_j \mathcal{P}_j^{t_j/3},$ then the order of $\langle\mathcal{I}\rangle$ is $3$ if $(\mathcal{P}\prod_j 
\mathcal{P}_j^{t_j/3})$ is not principal in $\mathcal{O}_K.$ To show $(\mathcal{P}\prod_j \mathcal{P}_j^{t_j/3})$ is not principal in $\mathcal{O}_K,$ it is sufficient to show $2\alpha$ is not a cube in $\mathcal{O}_K.$ Let $d'$ be the square-free part of $1-2m^3$ and denote
\begin{equation}\label{eqa}
1-2m^3=t^2d'\quad (t\in \mathbb{Z}).
\end{equation}
If $2\alpha=(a+b\sqrt{d'})^3$ for some $a,b\in \mathbb{Z},$ then we obtain 
\begin{equation}\label{eqb}
2=a^3+3ab^2d',
\end{equation}
\begin{equation}\label{eqc}
2|t|=3a^2b+b^3d'.
\end{equation}
From \eqref{eqb}, it holds $a|2$, that is, $a\in \{\pm1, \pm 2\}$.
Taking modulo $3$ in \eqref{eqb}, we see that $a\ne 1, -2$. When $a=-1$, we see from \eqref{eqb} and \eqref{eqc} that $d'=-1$ and $|t|=1$. Then by
\eqref{eqa}, we obtain $m=1$, which contradicts to the assumption
$m>1$. When $a=2$, we see from \eqref{eqb} and \eqref{eqc} that
$d'=-1$ and $|t|=11/2$. This is a contradiction.
This completes the proof.  
\end{proof}

\section{Numerical Examples}
In this section, we give some numerical examples corroborating  our results in \S 2 and \S 3. We compute the class numbers of each of the above families of fields for some small values of $d$ and list them in tables below. All the computations in this
paper were done using PARI/GP (version 2.7.6).
\vspace*{2mm}
\begin{table}[ht]
 \centering
\begin{tabular}{ | c | c | c | c |c|c|c|c|} 
 \hline
 $m$ & $  n$ &   $d=3(4m^{3n}-1)$ & $h(d)$ & $m$ &  $ n$ &   $d=3(4m^{3n}-1)$ & $h(d)$\\
\hline
 3 & 1 & 321 & 3& 3 & 2 & 8745 & 12\\
 3&3&2361953&36& 3&4&6377289&36\\
 9&1&8745&12& 9&2&6377289&36\\
 15&1&40497& 3& 21&1& 111129&6\\
 27&1&2361953&36& 33&1&431241&6\\
 39 & 1& 711825&3& 45 & 1& 1093497 & 3\\
 51&1&1591809&3& 57&1& 2222313&36\\
 63&1&30000561&9& 69&1&3942105&6\\
 75&1& 5062497&108& 81&1&6377289&36\\
 87&1&7902033&3& 93&1&9652281&6\\
 99&1&11643585&24& 105 & 1& 13891487 & 18\\
 111&1&16411569&6& 117&1&19219353&18\\
 123&1&22330401&12& 129&1&25760265&60\\
 135&1&29524497&18& 141 &1&33638649&60\\
 147&1&38118273&282& 153&1&42978921&9\\
 159&1&48236145&6&165&1&53905497&6\\
 \hline
\end{tabular}
\vspace*{2mm}
\caption{Numerical examples of Theorem \ref{thm2.1} for $k=1$ only.}
\end{table}

 \begin{table}
 \centering
\begin{tabular}{ | c | c | c | c |c|c|c|c|} 
 \hline
 m &   n &   $d=-(m^2n^2+4n)/3$ & $D=-(m^2n^2-4n)/3$ & $h(d)$ & $h(D)$\\
 \hline
 3&3&-31&-23&3&3\\
 3&15&-695&-655&24&12\\
3&21&-1351&-1295&24&36\\
3&33&-3311&-3223&72&30\\
3&39&-4615&-4511&36&84\\
3&51&-7871&-7735&120&48\\
9&3&-247&-239&6&15\\
9&15&-6095&-6055&84&36\\
9&21&-11935&-11879&72&150\\
9&33&-29447&-29359&132&72\\
9&39&-41119&-41015&120&180\\
9&51&-70295&-70159&252&168\\
15&3&-679&-671&18&30\\
15&15&-16895&-16855&96&84\\
15&21&-33103&-33047&60&150\\
21&3&-1327&-1319&15&45\\
21&15&-33095&-33055&240&72\\
21&21&-64855&-64799&120&222\\
27&3&-2191&-2183&30&42\\
27&15&-54695&-54655&216&156\\
\hline
\end{tabular}
\vspace*{2mm}
\caption{Numerical examples of Theorem \ref{thm2.2}.}
\end{table}
\vspace*{2mm}

\begin{table}
 \centering
\begin{tabular}{ | c | c | c | c |c|c|c|c|} 
 \hline
 m &   n & p & $d=-(3^mp^{2n}-2)$ & $D=-(3^mp^{2n}+4)$ & $h(d)$ & $h(D)$\\
 \hline
 3&1&3&-241&-247&12&6\\
 3&2&3&-2185&-2191&24&30\\
 3&1&7&-1321&-1327&24&15\\
3&2&7&-64825&-64831&24&162\\
 3&3&3&-19681&-19687&84&81\\
 5&1&3&-2185&-2191&24&30\\
 5&2&3&-19681&-19687&84&81\\
 5&1&5&-6073&-6079&24&57\\
 5&2&5&-151873&-151879&120&300\\
 7&1&3&-19681&-19687&84&81\\
 7&1&5&-107161&-107167&216&108\\
 \hline
\end{tabular}
\vspace*{2mm}
\caption{Numerical examples of Theorem \ref{thm2.3}.}
\end{table}
\vspace*{2mm}
\begin{table}
 \centering
\begin{tabular}{ | c | c | c | c |c|c|} 
 \hline
 $a$ & $b$ & $n$ & $d=3(a^{3n}-b^{2n})$ & $h(d)$\\
\hline
19&6&3&968062953369&6\\
19&21&3&967805794974&648\\
19&36&3&961532746329&24\\
19&51&3&915274229934&48\\
19&66&3&720101243289&12\\
19&81&3&120774483894&24\\
19&96&3&-1380210275751&1388160\\
19&111&3&-4643180563146&1951488\\
19&126&3&-11036449330791&4263624\\
19&141&3&-22606080831186&3780672\\
49&306&3&2422323582800979&192\\
\hline
\end{tabular}
\caption{Numerical examples of Theorem \ref{thm2.4}.}
\end{table}

\begin{table}
 \centering
\begin{tabular}{ | c | c | c | c |c|c|} 
 \hline
 $a$ & $b$ & $n$ & $d=3(4a^{3n}-b^{2n})$ & $h(d)$\\
\hline
1&3&2&-331&3\\
1&3&3&-725&6\\
1&3&4&-19671&84\\
1&9&2&-19671&84\\
1&9&3&-531437&480\\
1&15&2&-151863&324\\
7&3&2&1411545&12\\
7&15&2&1259913&6\\
7&27&2&-60845&192\\
13&9&2&57902025&72\\
13&15&2&57769833&48\\
13&21&2&57338265&96\\
\hline
\end{tabular}
\caption{Numerical examples of Theorem \ref{thm2.5}.}
\end{table}

\begin{table}
 \centering
\begin{tabular}{ | c | c | c | c |c|c|} 
 \hline
 $m$ & $d=-3(4m^3+1)$ & $h(d)$& $m$ & $d=-3(4m^3+1)$ & $h(d)$\\
\hline
3&-327&12 & 9&-8751&72\\
15&-40503&96 & 21&-111135&240\\
27&-236199&504& 33&-431247&360\\
39&-711831&648&45&-1093503&540\\
51&-1591815&780&57&-2222319&984\\
63&-3000567&1152&69&-3942111&2568\\
75&-5062503&1800&81&-6377295&1296\\
87&-7902039&2772&93&-9652287&1452\\
99&-11643591&2160&105&-13891503&2448\\
\hline
\end{tabular}
\caption{Numerical examples of Theorem \ref{thm3.1} (I).}
\end{table}

\begin{table}[ht]
 \centering
\begin{tabular}{ | c | c | c | c |c|c|} 
 \hline
 m & $d=1-2m^3$ & $h(d)$ & m & $d=1-2m^3$ & $h(d)$\\
\hline
 3&-53&6&5&-249&12\\
7&-685&12&9&-1457&24\\
11&-2661&48&13&-4393&24\\
15&-6749&84&17&-9825&12\\
19&-13717&48&21&-18521&228\\
23&-24333&72&25&-31249&96\\
27&-39365&180&29&-48777&96\\
31&-59581&126&33&-71873&240\\
35&-85749&336&37&-101305&144\\
39&-118637&216&41&-137841&264\\
43&-159013&162&45&-182249&408\\
47&-207645&288&49&-235297&264\\
51&--265301&684&53&-297753&228\\
55&-332749&318&57&-370385&360\\
59&-410757&336&61&-453961&420\\
63&-500093&78&65&-549249&624\\
67&-601525&66&69&-657017&660\\
71&-715821&672&73&-778033&324\\
75&-843749&1128&77&-913065&912\\
\hline
\end{tabular}
\caption{Numerical examples of Theorem \ref{thm3.2}.} 
\end{table}

\section*{Acknowledgements}
A. Hoque is supported by SERB N-PDF (PDF/2017/001958), Govt.
of India. The authors are indebted to the anonymous referee for his/her valuable suggestions which has helped improving the presentation of this manuscript.

\label{'ubl'}  
\end{document}